\documentclass[10pt,a4paper,twoside,leqno]{amsart}
\usepackage{amsthm,amsmath,amssymb}
\usepackage{a4wide}
\usepackage{color}
\usepackage{paralist}
\usepackage{amsmath,amsfonts,amsthm,amssymb,amsxtra,dsfont}
\usepackage{amsxtra, amssymb, mathrsfs, pstricks}
\usepackage[english]{babel}
\usepackage{amsmath,amsfonts,amstext,amsbsy,amsopn,amssymb,amsthm,amscd}
\usepackage{amsxtra,latexsym}
\usepackage{exscale}
\usepackage{graphicx,mathrsfs}
\usepackage{psfrag}
\usepackage{eucal,enumerate}
\usepackage{color,graphics}
\usepackage{extpfeil}
\newtheorem{teor}{Theorem}[section]
\newtheorem{defin}[teor]{Definition}
\newtheorem{lemm}[teor]{Lemma}
\newtheorem{osse}[teor]{Remark}
\newtheorem{remark}[teor]{Remark}
\newtheorem{prop}[teor]{Proposition}
\newtheorem{defi}[teor]{Definition}
\newtheorem{coro}[teor]{Corollary}
\newtheorem{prob}[teor]{Problem}

\newcommand{\bele}{\begin{lemm}\begin{sl}}
\newcommand{\enle}{\end{sl}\end{lemm}}
\newcommand{\bedef}{\begin{defi}\begin{sl}}
\newcommand{\eddef}{\end{sl}\end{defi}}
\newcommand{\bete}{\begin{teor}\begin{sl}}
\newcommand{\ente}{\end{sl}\end{teor}}
\newcommand{\beos}{\begin{osse}\begin{sl}}
\newcommand{\eddos}{\end{sl}\end{osse}}
\newcommand{\bepr}{\begin{prop}\begin{sl}}
\newcommand{\empr}{\end{sl}\end{prop}}
\newcommand{\bepro}{\begin{prob}\begin{rm}}
\newcommand{\empro}{\end{rm}\end{prob}}
\newcommand{\bede}{\begin{defin}\begin{sl}}
\newcommand{\edde}{\end{sl}\end{defin}}
\newcommand{\beco}{\begin{coro}\begin{sl}}
\newcommand{\enco}{\end{sl}\end{coro}}

\newcommand{\beeq}[1]{\begin{equation}\label{#1}}
\newcommand{\eddeq}{\end{equation}}

\newcommand{\beeqa}[1]{\begin{eqnarray}\label{#1}}
\newcommand{\eddeqa}{\end{eqnarray}}

\newcommand{\beal}[1]{\begin{align}\label{#1}}
\newcommand{\eddal}{\end{align}}

\newcommand{\bespl}[1]{\begin{split}\label{#1}}
\newcommand{\edspl}{\end{split}}

\newcommand{\bega}[1]{\begin{gather}\label{#1}}
\newcommand{\edga}{\end{gather}}

\newcommand{\beeqax}{\begin{eqnarray*}}
\newcommand{\eddeqax}{\end{eqnarray*}}

\def\qed{\ifmmode 
  \else \leavevmode\unskip\penalty9999 \hbox{}\nobreak\hfill
  \fi
  \quad\hbox{\hskip.5em\vrule width.4em height.6em depth.05em\hskip.1em}}
\def\endproofsym{\qed}
\renewenvironment{proof}[1][Proof]{\trivlist\item[\hskip\labelsep{\hskip0pt
    {\normalfont\scshape#1.}\hskip .321429\parindent}]\ignorespaces}
{\endproofsym\endtrivlist}
\def\endnobox{\def\endproofsym{}\end{proof}\def\endproofsym{\qed}}

\newcommand{\beeqao}{\begin{eqnarray}\no}
\newcommand{\bealo}{\begin{align}\no}
\newcommand{\besplo}{\begin{split}\no}
\newcommand{\begao}{\begin{gather}\no}

\newcommand{\dd}{\, \mathrm{d}}

\DeclareMathOperator{\dive}{div}

\newcommand{\calH}{{\cal H}}

\newcommand{\calR}{{\cal R}}

\newcommand{\uu}{{\bf u}}

\definecolor{ddmagenta}{rgb}{0.7,0,0.9}
\definecolor{ddorange}{rgb}{1,0.5,0}
\definecolor{Turk}{rgb}{0,0.7,0.4}

\newenvironment{mt}{\color{ddorange}}{\color{black}}
\newcommand{\bmt}{\begin{mt}}
\newcommand{\emt}{\end{mt}}

\newcommand{\R}{\mathbb{R}}
\newcommand{\N}{\mathbb{N}}
\newcommand{\bbK}{\mathbb{K}}
  
\def\calD{{\mathcal D}} \def\calE{{\mathcal E}} \def\calF{{\mathcal F}}
\def\calG{{\mathcal G}} \def\calH{{\mathcal H}} 
\def\calJ{{\mathcal J}}  \def\calL{{\mathcal L}}
  
 \def\calQ{{\mathcal Q}} \def\calR{{\mathcal R}}
\def\calS{{\mathcal S}}  
 \def\calW{{\mathcal W}} 
 \def\calZ{{\mathcal Z}}
\newcommand{\Sph}{\mathbf{U}}
\newcommand{\Spu}{\mathbf{H}}
\newcommand{\Spv}{\mathbf{V}}
\newcommand{\Spm}{\mathbf{M}}

\newcommand{\Spz}{\mathbf{Z}}
\newcommand{\Sps}{\mathbf{S}}
\newcommand{\Spx}{\mathbf{X}}

\numberwithin{equation}{section}
\begin{document}

\title[Gradient  system  in damage coupled with plasticity]{
A rate-independent gradient  system  in damage coupled with plasticity  via structured strains
}\thanks{The work of Elisabetta Rocca was
supported by the FP7-IDEAS-ERC-StG Grant \#256872
(EntroPhase), by GNAMPA (Gruppo Nazionale per l'Analisi Matematica, 
la Probabilit\`a e le loro Applicazioni) of INdAM (Istituto Nazionale di Alta Matematica). Elena Bonetti and Riccarda Rossi were
partially supported by a MIUR-PRIN'10-'11 grant for the project
``Calculus of Variations'',  and by GNAMPA (Gruppo Nazionale per l'Analisi Matematica, 
la Probabilit\`a e le loro Applicazioni)  of  
INdAM (Istituto Nazionale di Alta Matematica),  and  by IMATI -- C.N.R. Pavia. 
Marita Thomas was partially supported by 
the German Research Foundation (DFG) within CRC 1114, project C05. \\
 \textrm{{\bf Keywords:} rate-independent systems, tensorial damage model, 
anisotropic damage, plasticity, structured strain, energetic solutions, existence results. 
}
}
%
\author{Elena Bonetti}\address{Dipartimento di Matematica, Universit\`a degli Studi di Pavia, Via Ferrata~1, I-27100, Pavia, Italy. E-mail: elena.bonetti@unipv.it}
\author{Elisabetta Rocca}\address{Weierstrass Institute for Applied
Analysis and Stochastics, Mohrenstr.~39, D-10117 Berlin, Germany and
Dipartimento di Matematica, Universit\`a degli Studi di Milano, Via Saldini 50, I-20133 Milano, Italy.
E-mail: rocca@wias-berlin.de and elisabetta.rocca@unimi.it}
\author{Riccarda Rossi}\address{DIMI, Universit\`{a} di Brescia, Via Branze 38, I-25100 Brescia, Italy.
E-mail: riccarda.rossi@unibs.it}
\author{Marita Thomas}\address{Weierstrass Institute for Applied
Analysis and Stochastics, Mohrenstr.~39, D-10117 Berlin, Germany. E-mail: marita.thomas@wias-berlin.de}
%
%
\begin{abstract}
This contribution deals with a class of models combining 
isotropic damage with plasticity. 
It has been inspired by  a work by Freddi and Royer-Carfagni \cite{FRC}, including 
the case where the 
inelastic part of the strain only evolves in regions where the material is damaged. 
The evolution both of the damage and of  the plastic variable is assumed to be rate-independent. 
Existence of solutions is established in the abstract energetic framework elaborated 
by Mielke and coworkers (cf., e.g., \cite{Miel05ERIS, Miel08DEMF}).
\end{abstract}
\maketitle
\section*{Introduction}
 It is well known that damage in a material can be interpreted as a degradation of 
its elastic properties due to the failure of its microscopic structure.  
Such macroscopic mechanical effects take their origin from the 
formation of micro-cracks and cavities at a microscopic scale. 
Macroscopically, these degeneracy effects may be described by the incorporation of an internal  
variable into the model, the damage parameter, 
which in particular features a decrease of stiffness with ongoing damage.  
However, some materials show a more complex behavior, 
possibly presenting different responses to traction and compression loading, or exhibiting  
some  plastic-like  behavior when the 
damage process is activated.
\par
The study of plastic material behavior at small strains in itself has a long tradition, 
cf.\ e.g.\ \cite{Hill50,Lub90Plast}, 
and numerous analytical and numerical results exist, 
cf.\ e.g.\ \cite{Tema85MPP,HanRed99PMTN,RaDeGa08ADMB,Knee09SNGS,Knee10GSRC,
JiRoZe13LAVB,BaMiRo12QSSP,BaRo08TVPS,DMSca14QEPP,DMDSMoMo08GSQE,DMDSMo06QEPL}.
Also isotropic damage in itself nowadays is a well-investigated phenomenon and 
it has been treated in the spirit of phase-field  theories from 
the point of view of modeling, analysis,  and computations, cf.\ e.g.\ 
\cite{FreNed96DGDP,BoSch,bss,MiRou06,ThoMie09DNEM,FiKnSt10YMQS,KnRoZa13VVAR,
FraGar06VVPB,GarLar09TBQS,MiRoZe10CDEV,JirZem15LSRV,Giac05ATAQ,PhaMar13ODRC,DMIur13FMGL}. 
In this family of models,  a scalar internal variable is introduced 
to denote the local proportion of active micro-bonds vs.\ the damaged ones. 
Nonetheless, this approach does not permit to distinguish different 
anisotropic behaviors and the appearance of an unknown 
transformation strain, as it occurs in plasticity. Thus, it is of some interest to combine  
scalar and  tensorial variables to describe both of these two effects. 
\par
More precisely, in this contribution, we assume that a ``transition strain'', 
a \emph{structured strain} as it is called in \cite{FRC}, may appear and evolve during 
the damage process. The latter in itself decreases the stiffness of the material during the its evolution. 
The first effect makes our model akin to 
 a plasticity model, 
in which the plastic strain 
is activated through damage and its norm depends on the damage level; we refer to \cite{AlMaVi14GDMC}
for an alternative model for damage coupled with plasticity,  
recently analyzed in \cite{Crismale, Crismale-Lazzaroni}.  
As a consequence, we deal with two internal variables: a scalar one $\chi$, standardly 
denoting the local proportion of active bonds in the micro-structure of the material, 
and a tensorial one $D$, which stands for the transformation 
strain arising during the damage evolution. The behavior of these two variables is 
recovered by a generalization of the principle of 
virtual powers, in which micro-forces 
responsible for the  formation of micro-cracks and micro-slips  are included; 
we confine the discussion to 
the 
small-strain regime   and the isothermal case, though.
The momentum balance equation for the displacement $\mathbf{u}$ 
is eventually written in the quasi-static case, while 
the evolution of the internal variables $\chi$ and $D$ is governed by 
an energy functional and a $1$-homogeneous dissipation potential, leading to 
a rate-independent evolution of these variables and  possibly  including irreversibility constraints.  
\par
All in all, the resulting PDE system in the variable $\mathbf{q} = (\mathbf{u},\chi,D)$ 
pertains to the class of abstract gradient systems 
of the form
\begin{equation}
\label{gr_sys}
\partial\calR (\mathbf{q}_t) + \mathrm{D} \calE(t,\mathbf{q}(t)) \ni 0 \quad \text{in } (0,T),
\end{equation}
driven by an energy functional $\calE$ and a dissipation potential $\calR$, positively homogeneous of 
degree $1$ and only acting on the dissipative variables $(\chi,D)$. For the analysis of this system, 
we will resort to the \emph{energetic formulation} for rate-independent systems
developed by Mielke and coworkers, cf.\ \cite{MiTh04RIHM,Miel05ERIS,MaiMie05EREM,Miel08DEMF}. We will thus prove  the existence of 
energetic solutions by applying an abstract  existence result from \cite{Miel08DEMF}. 
\paragraph{\bf Plan of the paper.}
 The derivation of the model will be carried out in Section \ref{Conti}. 
The precise mathematical assumptions are collected in Section \ref{Ass}. 
The existence theorem (Thm.\ \ref{EnSolanisoEx}) is stated in Section \ref{FullyRI} 
in the framework of energetic solutions.  
Finally, its proof is carried out in Section \ref{AnaProp}. 
%
\section{Continuum mechanical derivation of the model}
\label{Conti}
%
Along a time-interval $[0,T],$ we study the mechanical behavior of a body, 
occupying a domain $\Omega\subset\R^d,$ $1<d\in\N$. The body is  exposed to 
time-dependent external loadings, which
possibly cause a  degradation of the micro-structure of the material, 
leading to  inelastic  responses. In particular, we restrict ourselves to a small-strain regime
and introduce the vector    $\uu$ of small displacements. 
Hence, as already mentioned in the introduction, we shall formulate the model 
in terms of the strain and in terms of two further state variables $\chi:[0,T]\times\Omega\to[0,1]$ 
and $D:[0,T]\times\Omega\to\R^{d\times d}$,  
which are internal variables  more specifically related to  the description of damage 
and plastic-like behavior.  
 Accordingly, in view of   the conjugate approach, the free energy will depend on the strain 
and on these two internal variables, and the stress shall be derived in terms of them. 
 \par
  In particular,
using the approach of \cite{FRC}, we first suppose that the symmetric gradient 
of the displacement $\uu$ 
$$
e(\uu)=(\nabla\uu+\nabla^\top\uu)/2
$$
is decomposed in two parts: 
\begin{equation}\label{symu}
e(\uu)=E_\mathrm{el}+\Xi
\end{equation}
where   $E_\mathrm{el} \in \R^{d\times d}$  represents the {\em elastic} part of the 
strain and   $\Xi\in \Sps\subset\R^{d\times d}$  the {\em inelastic} one, 
 associated with the formation of micro-cracks or micro-slips.  
 It is
 indeed 
 known that (see, e.g., 
 \cite{Kacha90ICDM,Fre02}) for an inelastic body the strain is determined by the stress 
and by some additional (internal) variable, which may be
interpreted within the framework of a general plasticity theory. 
In this spirit, the set 
$\Sps\subset\R^{d\times d}$  can, e.g., be the subspace of symmetric matrices 
or, as in plasticity theory, 
the subspace of  deviatoric  (i.e., trace-free) $\R^{d\times d}$-matrices. 
In order to allow for the treatment of different types of inelastic phenomena we keep 
$\Sps\subset\R^{d\times d}$ general, 
and refer to 
Remark \ref{comparison-with-FRC} below for more details on specific choices of $\Sps$ 
and their meaning.   
\par
 In the context of this damage model, 
 we prescribe that the inelastic part of the strain depends on the state of the internal 
bonds acting at a microscopic level in the material.
We also assume that the phenomenon of damage is {\em progressive}, in the sense that within 
the same body there may be regions where the material is completely damaged and  
regions  where the microstructure is lost, but not yet failed. 
As it is common in the modeling of isotropic damage, 
the variable $\chi$ is therefore linked to the proportion of active or  
inactive bonds in a neighborhood of material-dependent size (representative volume element) centered around
any  material point $x\in\Omega$. Hence, $\chi$ takes values in the interval $[0,1]$. 
Throughout this work we will assume that $\chi$ stands for   
 the proportion of \emph{active bonds} at the micro scale in the material, thus,  with  the  value $1$ 
in the sound regions and   $0$ in a failed zone. 
 Along the footsteps of  \cite{FRC} (cf.\ Remark \ref{comparison-with-FRC} later on), 
we  introduce a second internal variable 
$D \in \Sps$  of type ``transformation'' strain 
leading to plastic effects and developing in the regions where the material is damaged; 
it 
shall hereafter be  formally  referred to as   {\em plastic 
strain}. 
Thus,  following \cite{FRC},  the inelastic part of the strain 
 is a function of $\chi$ and $D,$
\begin{equation}
\label{defXi}
\Xi:[0,1]\times\Sps
\to \Sps
\quad\text{ s.t. }\quad
\Xi(1,D)=0\text{ and }\Xi(0,D)=D\qquad\text{for every }D\in\Sps\,.
\end{equation}
As a particular choice for the function $\Xi$ one may consider 
\begin{equation}\label{eqD}
\Xi(\chi,D)=(1-\chi)D\,.
\end{equation}
But as a general feature of $\Xi,$ note that, in view of \eqref{symu}, for  $\chi=1$ we have $E_\mathrm{el}=e(\uu)$, 
whereas for $\chi=0$ we have   $E_\mathrm{el}=e(\uu) - D$. 
\par
Following the continuum-mechanical modeling perspective of Fr\'emond, cf.\ e.g.\ \cite{Fre02}, 
we shall now introduce the constitutive functionals and equations specifying  
the damage-plasticity model under consideration. 
 Let us point out that this approach is mainly based on a variational principle, i.e.\
the (generalized) principle of virtual powers. 
The main idea is that forces acting at a microscopic level in the material, 
responsible for the formation of micro-cracks and thus activating damage, have to be included in
the whole energy balance of the mechanical system. Hence, as 
 a prerequisite we shall postulate that the powers of the interior forces $P_i,$ 
the exterior forces $P_e,$ and the acceleration forces $P_a,$   
acting on the elasto-plastic and damageable body, occupying the (reference) domain 
$\Omega\subset\R^d,$ are balanced, i.e.,  
\begin{equation}
\label{lavirt}
P_i+P_e=P_a\qquad\text{ and, here, }P_a=0\,,
\end{equation}   
as we will confine our discussion to a quasistatic evolution. 
%
\paragraph{\bf The principle of virtual powers. }
%
The principle of virtual powers 
now postulates that the above balance of powers has to 
hold on any subdomain $\omega\subset\Omega,$ thus leading to the 
\emph{virtual powers} of this subdomain, which are assumed to be given in integral form  by 
\begin{equation*}
P_e(\omega)=\int_\omega p_e\,\mathrm{d}x+\int_{\partial\omega}\widetilde p_e\,\mathrm{d}S
\qquad\text{and}\qquad 
P_i(\omega)=\int_\omega p_i\,\mathrm{d}x. 
\end{equation*}
Consequently, 
different kinds of virtual velocities are introduced: 
macroscopic velocities ${\bf v}$, microscopic scalar velocities $\gamma$, 
and microscopic tensorial velocities $V$. Under the assumption that no external 
forces act on the microscopic level, we can prescribe  
the virtual external power $P_e$ of the subdomain $\omega\subset\Omega$  in  the form 
\begin{equation}
\label{PE}
P_e(\omega)=
\int_{\omega}{\mathbf f}\cdot{\bf v}\,\mathrm{d}x
+\int_{\partial{\omega}}{\mathbf t}\cdot{\mathbf v}\,\mathrm{d}S\,
\end{equation}
for any macroscopic virtual velocity ${\mathbf v}:\omega\to\R^d$ and for the given volumetric 
force $\mathbf f:\omega\to\R^d$ and the given surface force $\mathbf t:\partial\omega\to\R^d$ 
acting on  
$\omega\subset\Omega$. 
Similarly, the virtual internal power of $\omega$ is given in integral form as the 
product of the internal 
forces and virtual velocities. Due to the fact that the body is exposed to elasto-plastic 
defomations and damage, the internal forces consist of the macroscopic stress  
$\sigma:\Omega\to\R^{d\times d}$ and 
additional internal micro-stresses $B:\omega\to\R,$ ${\bf H}:\omega\to\R^d,$ 
$X:\omega\to\R^{d\times d}$, and  ${\bf Y}:\omega\to\R^{d^3}$ 
related to  damage and the plastic deformation.    
 In what follows, the symbols $:$ and $::$ stand for the products 
in the spaces of $d^2$- and $d^3$-tensors, respectively. 
For any macroscopic virtual velocity ${\mathbf v}$, 
for all microscopic velocities $\gamma:\omega\to\R$, 
and for all microscopic tensorial velocities $V:\omega\to\Sps,$ 
the virtual internal power of $\omega$ is thus given by   
\begin{equation}\label{PI}
P_i(\omega)=\int_{\omega} p_i\,\mathrm{d}x=-\int_{\omega}(\sigma: e({\mathbf v})
+B\gamma+{\bf H}:\nabla\gamma+X:V+{\bf Y}::\nabla V)\,\mathrm{d}x\,.
\end{equation}
 Observe that the definition of $p_i$ reflects  the fact that 
the power of the interior forces is zero for any  (macroscopic)  rigid motion. 
\par
Now, taking into account that the relations \eqref{lavirt}--\eqref{PI} shall hold for 
any subdomain $\omega\subset\Omega$ and for any virtual velocity, the  
resulting balance equations are
\begin{subequations}
\label{balanceeqs}
\begin{alignat}{4}
\label{momentum}
-\hbox{div }\sigma&=&{\bf f}&\quad\text{ in }\Omega, 
\qquad \sigma{\bf n}&=&{\bf t}\quad\hbox{ on }\partial\Omega\,,\\
\label{micro}
B-\hbox{div }{\bf H}&=&0&\quad\text{ in }\Omega,
\qquad {\bf H}\cdot{\bf n}&=&0\quad\hbox{ on }\partial\Omega\,,\\
\label{microten}
X-\hbox{div }{\bf Y}&=&0&\quad\text{ in }\Omega,
\qquad{\bf Y}\cdot{\bf n}&=&0\quad\hbox{ on }\partial\Omega\,.
\end{alignat}
\end{subequations}
%
\paragraph{\bf The constitutive relations.}
%
Following \cite[Chap.s 3, 4]{Fre02}, we assume that the constitutive  relations are comprised 
in two functionals, the free enery functional $\calF$ and the pseudo-potential 
of dissipation $\calR$ in integral form, with densities $\Psi$ and $\Phi,$ respectively: 
\begin{equation}
\calF(\uu,\chi,D):=\int_\Omega\Psi\,\mathrm{d}x\qquad\text{and}\qquad
\calR(
\dot\chi,\dot D):= \int_\Omega\Phi\,\mathrm{d}x\,. 
\end{equation}
Formally using the above localization arguments, and in view of \eqref{balanceeqs}, we prescribe 
the following constitutive relations 
\begin{align}
\label{constrel}
\sigma=\frac {\partial\Psi}{\partial e},\;\;\;
B=\frac{\partial\Phi}{\partial\chi_t}+\frac{\partial\Psi}{\partial\chi},\;\;\;
{\bf H}=\frac{\partial\Psi}{\partial\nabla\chi},\;\;\;
X=\frac{\partial\Phi}{\partial D_t}+\frac{\partial\Psi}{\partial D},\;\;\;
{\bf Y}=\frac{\partial\Psi}{\partial \nabla D}\,.
\end{align}
%
\paragraph{\bf Choice of the constitutive functions.}
%
We choose the density of the pseudo-potential of dissipation of the form
\begin{equation}
\label{phi}
\begin{split}
\Phi(D_t,\chi_t)&:=
R_\mathrm{inel}(D_t)+R_\mathrm{dam}(\chi_t),\qquad\text{where }\\
R_\mathrm{inel}(D_t)&:=\mu|D_t|\;\text{ and }\;
R_\mathrm{dam}(\chi_t):=\nu|\chi_t|+I_{(-\infty, 0]}(\chi_t)\,
\end{split}
\end{equation}
 for material parameters $\mu,\nu>0$. 
Note that both $R_\mathrm{inel}$ and $R_\mathrm{dam}$ are positively $1$-homogeneous, 
thus featuring a rate-independent evolution of the 
variables $D$ and $\chi$. In particular, for  
the definition of $R_\mathrm{dam},$ observe that the indicator term $I_{(-\infty,0]}$
enforces the unidirectionality constraint 
$\chi_t \leq 0$, i.e.\ that the parameter $\chi$ is a non-increasing function of time, 
 ensuring 
that damage can only increase. In turn, the free energy density $\Psi$ 
shall feature an indicator term acting on 
$\chi,$ cf.\ \eqref{choiceJ}, which forces the $\chi$-component of the solution to the 
evolutionary system we shall derive to take positive values. 
Starting from an initial datum $\chi_0$ with $\chi_0(x) \in [0,1]$ for 
almost all $x$ in $\Omega$, we will thus obtain that $\chi(x,t) \in [0,1]$ for almost all 
$(x,t)\in \Omega \times [0,T],$  in accordance with the physical meaning 
of $\chi$ as a \emph{proportion} of active bonds. 
\par 
For the definition of the free energy density $\Psi$ we assume, 
in the spirit of linear elasticity, that $\Psi$   
consists of a quadratic  elastic  contribution, a coupling term $H$, and 
 terms $J$ and $G$ also featuring regularizations for 
the damage variable and the plastic strain, respectively. In particular, $\Psi$ shall take 
the following form:  
\begin{equation}
\label{psi1}
\begin{split}
\Psi(e(\uu),\chi,D,\nabla\chi,\nabla D)
:=&\tfrac{1}{2}\big(e(\uu)-\Xi(\chi,D)\big):\bbK(\chi):\big(e(\uu)-\Xi(\chi,D)\big)\\
&\quad+H(\chi,D)+J(\chi,\nabla\chi)+G(D,\nabla D)\,.
\end{split}
\end{equation}
Here, $\Xi:\R\times\R^{d\times d}\to\R^{d\times d}$ is defined as in \eqref{defXi}. 
Moreover, as usual in damage models we consider a $\chi$-dependent stiffness tensor 
$\bbK $,  such that  $\bbK(\chi)$ is a symmetric  
$\R^{d^4}$-tensor for every $\chi,$ 
with the property that a decrease of the value of $\chi$ leads to a decrease of 
the quadratic contribution of the energy term.  
 In principle, with the choice of $\bbK$ we can incorporate in the model 
both the case in which the stiffness  degenerates when the material is completely damaged 
(i.e., for $\chi=0,$ the tensor $\bbK(0)$ is no longer positive definite, 
in particular it might happen that $\bbK(0)=0,$  cf.\ e.g.\ 
\cite{BouMiRou07,MiRoZe10CDEV,Miel11CDEE,HeiKra15CDLE}),  
and the case in which
some residual stiffness is guaranteed   even for $\chi =0$  (i.e.,  
$K(\chi)$ is positive definite for every $\chi$). 
In fact, in what follows we shall confine our analysis to the latter case. 
\par 
The coupling term $H$ shall take into account 
different cohesive properties of the material and the plastic behavior.
A possible choice could be 
\begin{equation}
\label{choiceH}
 H(\chi,D)=w(1-\chi)+\tfrac{1}{2}|D|^2(1-\chi)\qquad\text{with }w>0\,.
\end{equation}
Observe that, 
 if the material is completely undamaged, i.e.\ $\chi=1,$ 
the term $H$ plays the role of a cohesion energy while, in the case $\chi=0,$ i.e.\ 
when the material is maximally ``broken'', it leads to a hardening effect 
for the plasticity variable $D$, since $w>0$.   
\par 
The  function   $J$ for the damage variable shall guarantee the modeling  assumption  
$\chi\in[0,1]$ a.e.\ in $\Omega$  introducing some internal constraint. A possible choice would be 
\begin{equation}
\label{choiceJ}
J(\chi,\nabla\chi):=I_{[0,1]}(\chi)+\tfrac{\alpha}{2}|\nabla\chi|^2
+W_1(\chi)\,.
\end{equation}
\begin{osse}
\upshape
Let us point out that, setting $W_1(\chi)=\tfrac{1}{\alpha}(1-\chi)^2$ with small $\alpha>0$ 
would rather inhibit damage.  
However, we may allow also for non-convex choices of $W_1,$ e.g.\ in terms of 
a double-well potential. 
As analyzed in \cite{Thom11QEBV} in the context of brittle damage, 
the choice $W_1(\chi)=\tfrac{1}{\alpha}\chi^2(1-\chi)^2$ 
with small $\alpha>0$ will yield  
that $\chi\in\{0,1\}$ a.e.\ in $\Omega$ as $\alpha\to0,$ thus accounting only for the 
sound and the maximally damaged state of the material in the limit.
\end{osse}
In the same manner, the  term  $G$ may confine the plastic strain to a 
closed, convex subset $K$ of the subspace $\Sps\subset\R^{d\times d}$; 
we refer to Remark \ref{comparison-with-FRC} for different choices of $K$ and $\Sps$. 
As a possible form of $G$ we may consider 
\begin{equation}
\label{choiceG}
G(D,\nabla D):=I_K(D)+(|D|^2-1)^2+\tfrac{1}{q}|\nabla D|^q \qquad\text{ with }q\in(1,\infty)\,.
\end{equation}
\begin{remark}[Comparison with \cite{FRC} and possible choices of $\Sps$ and $K$]
\label{comparison-with-FRC}
\upshape
 In \cite{FRC}, following 
\cite{DP-Owen}  the authors refer to $\Xi$ as the \emph{structured strain}, 
and postulate for it the form  \eqref{eqD} as a function of 
$D$, which, as a function of $x\in \Omega$,  in 
turn represents the \emph{structured strain} that would develop in a neighbourhood of 
$x$ if the material was completely disgregated. 
In \cite{FRC} it is in fact remarked that the form of the field $D=D(x)$ may depend upon 
the material microstructure and the local defects of the body, so that
its complete characterization is an open problem. Hence, the authors propose 
a \emph{mesoscopic representation} for $D$ as a function of 
 $e(\uu)$.   
\par
 More precisely, the relation between $D$ and $e(\uu)$ is established through the minimization 
of  the quadratic elastic energy
 (i.e.\ the first term in \eqref{psi1}), 
for $\chi=0$, over suitable classes of admissible structured strains. 
This reflects the fact that $D$ physically represents the 
strain that a completely disgregated body may attain without energy consumption in order 
to accomodate the boundary data.
Different choices for the class  $\Sps$ of admissible strains lead to models with 
different types of material responses to damage and fractures.
\par
For example, taking  $\Sps$  as the space of  all symmetric  
tensors  it is possible to recover a model describing the formation of 
\emph{cleavage fractures}, viz.\ fractures directly proportional to the 
macroscopic deformation. 
Indeed, minimizing the  elastic contribution to the  free energy in \eqref{psi1}  
for $\chi=0$,   in the case in which  
$\bbK(0)$ is positive definite, leads to     $D=e(\uu)$. Observe that, 
when 
$D=e(\uu)$ we recover the original form of the elastic part in the free energy
\[
 \frac{1}{2}\big(e(\uu)-(1-\chi)D\big):\bbK(\chi):\big(e(\uu)-(1-\chi)D\big)=\frac 1 2 \chi^2|e(\uu)|^2.
\]
\par 
Setting  $\Sps$ as the space of symmetric tensors with null deviatoric part  
(i.e., \emph{trace-free} matrices) brings to a model for the formation 
of less brittle fractures, like those occurring in materials like stones.  
In this connection, as common in plasticity models (cf.\ e.g. \cite{HanRed99PMTN}),  
we might choose   $K\subset\Sps$ as a closed and convex subset of the set 
of deviatoric  matrices $\Sps$. 
  \end{remark}
%
\paragraph{\bf The final set of constitutive equations.}
%
Combining in \eqref{momentum}-\eqref{microten} 
the constitutive relations \eqref{constrel} with \eqref{phi} and \eqref{psi1} we obtain the 
set of constitutive equations, to be satisfied in $\Omega \times (0,T)$: 
\begin{subequations}
\label{formal-system}
\begin{eqnarray}
&&-\dive\left(\bbK(\chi):\big(e(\uu)-\Xi(\chi,D)\big)\right)={\bf f}\,,\\[2mm] %
&&\begin{aligned}
&
\partial R_\mathrm{dam}(\chi_t)+\partial_\chi J(\chi,\nabla\chi)
-\dive\tfrac{\partial J(\chi,\nabla\chi)}{\partial(\nabla\chi)}\\
&\quad\ni-\tfrac{1}{2}\big(e(\uu)-\Xi(\chi,D)\big):\bbK'(\chi):\big(e(\uu)-\Xi(\chi,D)\big)\\
&\qquad 
+\big(e(\uu)-\Xi(\chi,D)\big):\bbK(\chi):\tfrac{\partial\Xi(\chi,D)}{\partial\chi}
-\tfrac{\partial H(\chi,D)}{\partial\chi}\,,
\end{aligned}
\\[2mm]
&&
\begin{aligned}
&
\partial R_\mathrm{inel}(D_t)+\partial_D G(D,\nabla D)
-\dive\tfrac{\partial G(D,\nabla D)}{\partial(\nabla D)} \\
&\quad-\bbK(\chi)\left(e(\uu)-\Xi(\chi,D)\right):\tfrac{\partial\Xi(\chi,D)}{\partial D}
+\tfrac{\partial H(\chi,D)}{\partial D}\ni0\,
\end{aligned}
\end{eqnarray}
\end{subequations}
 with $\mathbf{f} $ the volume force  from  \eqref{PE}.
We shall assume that the inelastic stress function $\Xi$ and the material tensor $\mathbb{K}$ 
are suitably smooth.  
 We will supplement the rate-independent  system 
\eqref{formal-system}
 with the boundary conditions 
 \begin{equation}
 \label{b-c}
\begin{split}
& \uu(x,t)=\uu_D(t)\text{ on }  \Gamma_D,\qquad
\mathbb{K}(\chi):\big(e(\uu(x,t))-\Xi(\chi,D))\mathrm{n}=\mathbf{t}\text{ on }\Gamma_N\,,\\ 
&\tfrac{\partial\chi}{\partial\mathrm{n}}=0 \text{ in  } \partial\Omega \times (0,T), 
\qquad \tfrac{\partial D}{\partial\mathrm{n}}=0 \text{ in  } \partial\Omega \times (0,T),
\end{split}
\end{equation}
with $\Gamma_\mathrm{D}$ and $\Gamma_\mathrm{N}$ the Dirichlet and the Neumann parts of the boundary 
$\partial\Omega$, 
respectively, and $\mathrm{n}$ the outward unit normal to $\partial\Omega$. 
 In what follows, we will address the existence of solutions to the 
boundary-value problem \eqref{formal-system}--\eqref{b-c} in  a suitably weak sense. 
We will discuss a suitable solution concept in Section \ref{FullyRI} ahead. 
 \begin{remark}
 \upshape
 \label{rmk:particular-case}
Let us point out that system \eqref{formal-system} is related, for special choices of the involved functionals, 
to well-known models in plasticity and phase transitions processes. Indeed, taking, for example, 
$$
\bbK(\chi)=\mathrm{Id}\in\R^{d^4}, 
\quad  \Xi(\chi,D) = (1-\chi)D,  \quad\hbox{and}\quad H(\chi,D)=w(1-\chi) + \frac 1 2 |D|^2 (1-\chi),
$$
we get  the following PDE system in $\Omega \times (0,T)$
\begin{subequations}
\label{formal-system-simpler}
\begin{eqnarray}
&&-\dive\left(e(\uu)-(1-\chi)D \right)={\bf f}\,,\\[2mm] %
&&\begin{aligned}
&
\partial R_\mathrm{dam}(\chi_t)+\partial_\chi J(\chi,\nabla\chi)
-\dive\tfrac{\partial J(\chi,\nabla\chi)}{\partial(\nabla\chi)}\\
&\quad\ni-\tfrac{1}{2}\big| e(\uu)-(1-\chi)D\big|^2  -D\colon(e(\uu) -(1-\chi)D) + \frac12 |D|^2 +w,
\end{aligned}
\\[2mm]
&&
\begin{aligned}
&
\partial R_\mathrm{inel}(D_t)+\partial_D G(D,\nabla D)
-\dive\tfrac{\partial G(D,\nabla D)}{\partial(\nabla D)} \\
&\quad-\left(e(\uu)-(1-\chi) D)\right)(1-\chi)
+(1-\chi)D\ni0\,.
\end{aligned}
\end{eqnarray}
\end{subequations} 
In particular, 
 without 
terms as $J$ and $G$ in the free energy, the resulting equations correspond to a 
rate independent evolution for the parameter $\chi$, governed by a quadratic source of damage 
(including strain and structured/plastcity strain), and a plasticity equation with hardening 
contribution (obtained by the third equation in the case $\chi=0$). 
\end{remark}
%
\section{Assumptions and notation}
\label{Ass}
%
In the following, given a Banach space  $\mathbf{B}$,  we shall denote 
by $\mathbf{B}^*$ its dual space, by $\|\cdot\|_\mathbf{B}$ its norm,  
and by $\langle \cdot, \cdot \rangle_\mathbf{B}$ 
the duality pairing between $\mathbf{B}^*$ and $\mathbf{B}$.  
We set $R_\infty: = \R \cup \{+\infty\}$. 

In the next lines, we specify the mathematical assumptions for the quantities introduced so far. 
\par
%
\paragraph{\bf Assumptions on the domain: } 
%
We assume that
\begin{equation}
\begin{split}
\label{ass-dom}
&\Omega\subset\R^d\,,\;d\in\N\,,\text{ is a bounded domain with Lipschitz-boundary
$\partial\Omega$ such that\ }\\
&\Gamma_\mathrm{D}\subset\partial\Omega\text{ is nonempty and relatively open and }
\Gamma_\mathrm{N}:=\partial\Omega\backslash\Gamma_\mathrm{D}\,.
\end{split}
\end{equation}
%
\paragraph{\bf Function spaces: }
%
We fix the function spaces as follows  
\begin{subequations}
\label{spaces}
\begin{eqnarray}
\label{Sph}
\Sph &:=&\{u\in H^1(\Omega,\R^d),\,u=0\text{ on }\Gamma_D\}\,,\\
\label{Spz}
\Spz&:=&L^1(\Omega)\,,\\
\label{Spm}
\Spm&:=&\{z\in\Spz,\,z\in[0,1]\text{ a.e.\ in }\Omega\}\,,\\
\label{Spx}
\Spx&:=&\{z\in W^{1,r}(\Omega)\}\,, \qquad r>1,  \\
\label{Spv}
\Spv&:=& L^1(\Omega;\R^{d\times d})\,,  \\
\label{Spu}
\Spu &:=&W^{1,q}(\Omega;\Sps)\,, \qquad q\geq2d/(d\!+\!2)
,\\
\label{Spu1}
\Spu_1&:=&L^{q_1}(\Omega;\Sps)\,, \qquad q_1 \geq \max\{q,2\}, 
\end{eqnarray}
and we recall that $\Sps$ is a subspace of $\R^{d\times d}$.  With \eqref{growthG}
ahead we shall further specify the conditions on the  indices $q$ and $q_1$. 
\end{subequations}
%
\paragraph{\bf Assumptions on the given data: }
%
Given 
$\Sph$ from \eqref{Sph}, we shall 
assume that the volume forces $\mathbf{f}$ and the surface forces $\mathbf{t}$ from \eqref{PE}  
are comprised in   a  time-dependent functional $F: [0,T] \to \Sph^*$. 
Moreover, for all $t\in[0,T],$ we suppose that the Dirichlet datum 
$\uu_\mathrm{D}(t)$ has has an extension 
from $\Gamma_\mathrm{D}$ into the domain $\Omega,$ also denoted 
by $\uu_\mathrm{D}(t)$.  In particular, we make the following regularity assumptions:   
\begin{subequations}
\label{given}
\begin{eqnarray}
&&
\begin{aligned}
&F\in C^1([0,T];\Sph^*)\text{ comprises both volume forces and Neumann data,}\\
&\text{such that }\|F\|_{C^1([0,T];\Sph^*)}\leq C_F,\\ 
\end{aligned}
\\
&&
\begin{aligned}
&\uu_\mathrm{D}\in C^1([0,T];\Sph)\text{ is an extension of the Dirichlet datum,}\\
&\text{such that }\|\uu_\mathrm{D}\|_{C^1([0,T];\Sph)}\leq C_\mathrm{D}
\text{ and }e_\mathrm{D}:=e(\uu_\mathrm{D})\,.\\
\end{aligned} 
\end{eqnarray}
Furthermore, for the elastic tensor $\mathbb{K}:[0,1]\to\R^{d^4}$ 
we assume symmetry and positive definiteness, i.e.,  
\begin{equation}
\label{assK}
\begin{split}
&\forall\,\chi\in[0,1]:\;\mathbb{K}(\chi)\text{ is symmetric}\,, 
\\
&\exists\,K_1,K_2>0\,\forall\,e\in\R^{d\times d}\,\forall\,\chi\in[0,1]:\;
|e|^2K_1\leq e:\bbK(\chi):e\leq K_2|e|^2\,. 
\end{split}
\end{equation}
\end{subequations}
Recall that the positive definiteness of $\mathbb{K}(\chi)$ for any $\chi\in[0,1]$ 
ensures some residual stiffness 
of the material even in the case of maximal damage $\chi=0$.
%
\paragraph{\bf Assumptions on the inelastic strain $\Xi$: }For the inelastic strain 
%
$\Xi:[0,1]\times\R^{d\times d}\to\R^{d\times d}$ introduced in \eqref{defXi} we make 
the following assumptions:  
\begin{subequations}
\label{propXi}
\begin{eqnarray}
\label{CharaXi}
\hspace*{3mm}
&&\Xi\in C^0([0,1]\times \Sps;\Sps) \quad\text{s.t.}\\
\label{monoXi}
\hspace*{3mm}
&&\forall\,\chi_1<\chi_2\in[0,1],\,D\in \Sps 
:\;|\Xi(\chi_2,D)| \leq |\Xi(\chi_1,D)|\,.\;
\Xi(1,D)=0\text{ and }\Xi(0,D)=D\,.\quad\,
\end{eqnarray}
\end{subequations} 
For later use we remark that \eqref{monoXi} in particular implies that 
\begin{equation}
\label{growthXi}
|\Xi(\chi,D)| \leq |D| \qquad \text{for all } (\chi,D)\in [0,1] \times \Sps\,.
\end{equation}
%
\paragraph{\bf Assumptions on the damage regularization: } 
%
In view of \eqref{psi1}, given 
$\Spz$ from \eqref{Spz}, we define the damage 
regularization functional in terms of   
\begin{subequations}
\label{propJ}
\begin{equation}
\label{defJ}
\begin{split}
&\calJ:\Spz\to\R_\infty,\;\calJ(\chi):=\int_\Omega J(\chi,\nabla\chi)\,\mathrm{d}x
\quad\text{with}\\
&J(\chi,\nabla\chi):=I_{[0,1]}(\chi)+\tilde J(\chi,\nabla\chi)\, 
\end{split}
\end{equation}
and we assume that $\tilde J$ has the following properties:
\begin{eqnarray}
\label{contiJ}
\hspace*{-8mm}
\text{Continuity:}\hspace*{-5mm}&&\tilde J\in C(\R\times\R^d;\R)\,,\\
\nonumber
\hspace*{-8mm}
\text{Growth:}\hspace*{-5mm}&&\exists\, c_J,\tilde c_J,C_J>0,\,\exists\,r\in(1,\infty),\,
\forall\chi\in[0,1],\,A\in\R^d:\\
\label{growthJ}
&&\hspace*{20mm}c_J(|A|^r-\tilde c_J)\leq\tilde J(\chi,A)\leq C_J(|A|^r+1)\,,
\qquad
\\ 
\hspace*{-8mm}
\label{convexJ}
\text{Convexity:}\hspace*{-5mm}&&\forall\chi\in[0,1]:\;\tilde J(\chi,\cdot)
\text{ is convex on }\R^{d\times d}\,.
\end{eqnarray}
\end{subequations}
Under these assumptions $\calJ$ may e.g.\ be a pure (convex) gradient 
regularization (i.e.\ $\tilde J$ does not 
depend on $\chi$), but it may also incorporate terms like $(1-\chi)^2$ enforcing $\chi$ 
to stay close to $1,$ hence inhibiting damage. Also nonconvex terms of lower order, 
e.g.\ double well potentials, may contribute to $\tilde J$, provided the leading term is convex. 
In particular, note that the density $J$ 
considered in \eqref{choiceJ} is comprised in this set of assumptions.      
%
\paragraph{\bf Assumptions on the plastic regularization: }
%
In view of \eqref{psi1}, given 
$\Spv$ from \eqref{Spv} and the subspace 
$\Sps\subset\R^{d\times d}$, 
 we introduce the plastic 
regularization functional as follows  
\begin{subequations}
\label{propG}
\begin{equation}
\label{defG}
\begin{split}
&\calG:\Spv\to\R_\infty,\;\calG(D):=
\begin{cases}
\int_\Omega G(D,\nabla D) \dd x & \text{if } G(D,\nabla D) \in L^1(\Omega),
\\
\infty & \text{otw.}
 \end{cases} 
\quad\text{with}\\
&G(D,\nabla D):=I_{K}(D)+\tilde G(D,\nabla D)\,, 
\end{split}
\end{equation}
where $K$ is a closed, convex subset of $\Sps,$   
and we suppose that $\tilde G$ has the following properties:
\begin{eqnarray}
\label{contiG}
\hspace*{8mm}
\text{Continuity:}\hspace*{-5.5mm}
&&\tilde G\in C(\R^{d\times d}\times\R^{d^4};\R)\,,\\
\hspace*{8mm}
\nonumber
\text{Growth:}\hspace*{-5.5mm}&&\exists\, c_G,\tilde c_G,C_G>0,\,\exists\, q,\, q_1 \text{ with } 
\tilde q: =\tfrac{2d}{d+2}\leq q<q_1\in[2,\infty),
\forall D\!\in\!\R^{d\times d},
A\!\in\!\R^{d^4}\!:
\\
\label{growthG}
&&\hspace*{5.5mm}
c_G(|A|^q+|D|^{q_1}-\tilde c_G)\leq\tilde G(D,A)\leq C_G(|A|^q+|D|^{q_1}+1)\,,\qquad\\ 
\hspace*{-8mm}
\label{convexG}
\text{Convexity:}\hspace*{-5.5mm}&&\forall\,D\in\R^{d\times d}:\;\tilde G(D,\cdot)
\text{ is convex on }\R^{d^4}\,.
\end{eqnarray}
\end{subequations}
The assumption on the exponents $q<q_1$ 
(which are the ones associated with the spaces $\Spu$ and $\Spu_1$, cf.\ \eqref{Spu} and \eqref{Spu1}), 
implies the  coercivity of the integral functional $\calG$ wrt.\ the space 
$W^{1,q}(\Omega;\R^{d\times d})$. Moreover, $q_1\geq 2$ yields that 
$D\in L^2(\Omega;\R^{d\times d})$ on energy sublevels, whereas the lower bound 
$\tfrac{2d}{d+2}\leq q$ ensures that 
\begin{equation}
\label{assq}
W^{1,q}(\Omega;\R^{d\times d})\Subset L^2(\Omega;\R^{d\times d})
\text{ compactly, with embedding constant }C_{\Spu\to L^2}\,. 
\end{equation}
Note that, for instance, $G$ from \eqref{choiceG} complies with the 
above growth assumptions if $q\in[\tilde q,4]$.
%
\paragraph{\bf Assumptions on the coupling term: }
%
In view of \eqref{psi1}, \eqref{defJ}, 
and \eqref{defG} we introduce the coupling term as follows 
\begin{subequations}
\label{defH}
\begin{equation}
\label{defH1}
\calH:\Spz\times\Spv\to\R\,\quad \calH(\chi,D):=
\left\{
\begin{array}{ll}
\int_\Omega H(\chi,D)\,\mathrm{d}x&\text{if }(\chi,D)\in\Spm\times(\Spu\cap\Spu_1),\\
\infty&\text{if }(\chi,D)\in\Spz\times\Spv,
\end{array}\right. 
\end{equation}
and  for the density $H$  we assume 
\begin{eqnarray}
\label{contiH}
\hspace*{8mm}
\text{Continuity:}\hspace*{-6mm}&&H\in C([0,1]\times\R^{d\times d};\R)\,,\\
\label{growthH}
\text{Growth:}\hspace*{-6mm}&&\exists\, C_H\!>0,\,\exists\,q_2\in[1,q_*)
\forall\,(\chi,D)\!\in\![0,1]\times\R^{d\times d}\!:\;
0\leq H(\chi,D)\leq C_H(|D|^{q_2}\!+\!1)\,,
\end{eqnarray}
\end{subequations}
 where $q_*=dq/(d-q)$ if $q<d$ and $q_*=\infty$ if $q\geq d$. Note that $q_2\in[1,q_1]$ would be 
sufficient to ensure the integrability of $H$. But for the continuity of $\calH$ it is required that 
$q_2\in[1,q_*)$. Also note 
that the special choice \eqref{choiceH} of $H$ complies with the above assumptions 
with   $q_2 \geq 2$. 
%
\section{Energetic solutions for the rate-independent system with damage and plasticity}
\label{FullyRI}
%
In view of the positively $1$-homogeneous character of the 
 pseudo-potential of dissipation $\calR$  from \eqref{phi}, system 
\eqref{formal-system} is rate-independent. 
Therefore, for the  analysis of the associated initial-boundary value problem  
we will resort to a weak solvability concept for rate-independent systems, 
namely the notion of energetic solution, 
cf.\  \cite{MiTh04RIHM,Miel05ERIS}. 
In order to give it in the context of the present system with damage and plasticity 
we now introduce the energy functional, depending on $t\in [0,T]$ 
 and  on the state variables $(\uu,\chi,D)$, 
  and the dissipation potential
   associated with 
\eqref{formal-system}. In accordance with Sec.\ \ref{Ass}
 we set 
\begin{subequations}
\label{defFctsfinal}
\begin{eqnarray}
\label{defD}
&&
\calR=\calR_\mathrm{inel}+\calR_\mathrm{dam}:\Spv\times\Spz\to\R_\infty,\quad\text{ where,}\\
\label{defV2}
&&
\calR_\mathrm{inel}: \Spv\to[0,\infty),\quad\calR_\mathrm{inel}(A):=
\int_\Omega R_\mathrm{inel}(A)\,\mathrm{d}x
\quad\text{with }R_\mathrm{inel}(A):=\mu|A|\,,\\
\label{defR}
&&
\calR_\mathrm{dam}:\Spz\to[0,\infty],\quad\calR_\mathrm{dam}(z)
:=\int_\Omega R_\mathrm{dam}(z)\,\mathrm{d}x\quad\text{with }
R_\mathrm{dam}(z):=\nu|z|+I_{(-\infty,0]}(z)\,,\\
\label{defE}
&&
\begin{aligned}
\calE:[0,T]\times\Sph\times\Spz\times\Spv\to\R_\infty\,,
\;\calE(t,\uu,\chi,D): & 
= \calF(\uu,\chi,D) -  \langle F(t),\uu\rangle_\Sph 
\\
&
=\calW(t,\uu,\chi,D)+\calJ(\chi)+\calG(D)+\calH(\chi,D)
\,,
\end{aligned}
\\
\label{defW}
&&
\begin{aligned}
&
\calW:[0,T]\times\Sph\times \Spm
\times L^{2}(\Omega;\R^{d\times d})\to\R\,,
\end{aligned}\\
&&
\begin{aligned}
\calW(t,\uu,\chi,D):=& \frac12  \int_\Omega
\big(e(\uu)\!  + \!e_\mathrm{D}(t)\!-\!\Xi(\chi,D)\big):
\mathbb{K}(\chi):\big(e(\uu)\!  +  \!e_\mathrm{D}(t)\!-\!\Xi(\chi,D)\big)
\,\mathrm{d}x\\
&-\langle F(t),\uu\rangle_\Sph\,. \qquad
\end{aligned}
\nonumber
\end{eqnarray}
\end{subequations}
In order to give the energetic concept we shall use the shorthand notation 
$\mathbf{q}=(\mathrm{u},\chi,D)$ and set $\calQ:=\Sph\times\Spz\times\Spv,$  
the state space where $\mathbf{q}$ varies, 
whereas $\calZ:=\Spz\times\Spv$ stands for the space of the \emph{dissipative} variables $z:=(\chi,D)$. 
In this way, we shall now state the definition 
of energetic solutions in an abstract form. 
\par 
\begin{defin}[Energetic formulation of rate-independent processes]
\label{defenfo}
\hspace*{-0.9mm}
For\,the\,initial\,datum\,$\mathbf{q}_0\!\in\!\calQ$ find $\mathbf{q}\!:\![0,T]\!\to\!\!\calQ$ such that for all  
$t\in[0,T]$ the global stability \eqref{gS} and the global energy balance \eqref{gE} hold
\begin{subequations}
\label{enfo}
 \begin{align}
\label{gS}
\text{\rm Stability}&:\quad\text{for all }\tilde{\mathbf{q}}\in\mathcal Q:\;\hspace*{2mm}
\mathcal E(t,\mathbf{q}(t))\leq\mathcal E(t,\tilde{\mathbf{q}})+\mathcal R(\tilde{\mathbf{q}}{-}\mathbf{q}(t)),\\
\label{gE}
\text{\rm Energy balance}&:\quad
\mathcal E(t,\mathbf{q}(t))+\mathrm{Diss}_{\mathcal R}(\mathbf{q},[s,t])=
\mathcal E(0,\mathbf{q}(0))+\int_s^t \partial_t  \mathcal E(\xi,\mathbf{q}(\xi))\,\mathrm{d}\xi
\end{align} 
\end{subequations}
with $\mathrm{Diss}_{\mathcal R}(\mathbf{q},[0,t]):=\sup
\big\{\sum_{j=1}^N\mathcal R(\mathbf{q}(\xi_{j}){-}\mathbf{q}(\xi_{j-1}))
\,|\,s=\xi_0<\ldots<\xi_N=t,\,N\in\N\big\}$, where $\calR (\mathbf{q}_1-\mathbf{q}_2)$ 
has to be understood as $\calR(z_1-z_2)$  with 
$\mathbf{q}_i = (\uu_i,z_i)$ for $i=1,2$. 
\end{defin}
The claim that \eqref{enfo} has to hold for all $t\in[0,T]$ 
entails that the energetic formulation is only solvable for initial data 
$\mathbf{q}_0$ which satisfy \eqref{gS} for $t=0.$ For later convenience we introduce 
the set of stable states at time $t\in[0,T]$ 
\begin{align}\label{stS}
\calS(t):=\{\mathbf{q}\in\calQ,\, \mathbf{q}\text{ satisfies }\eqref{gS} 
\text{ wrt. } \calE(t,\cdot) \text{ and }\calR\}\,. 
\end{align}
A solution in terms of the energetic formulation is called 
an energetic solution to the rate-independent system $(\calQ,\calE,\calR)$. 
\par
In what follows we will investigate the existence of energetic solutions \eqref{enfo}  
for the rate-independent system with damage and plasticity 
defined by  the functionals $\calE$ and $\calR$ from   \eqref{defFctsfinal} by verifying the 
assumptions of an abstract existence theorem given   
in \cite{Miel08DEMF}, cf.\ also \cite{MiTh04RIHM,Miel05ERIS, MRS06}. 
We now  shortly recap this result, highlighting the role of a 
series of conditions on the driving functionals 
$\calE$ and $\calR$ enucleated below, cf.\ \eqref{CE}--\eqref{CC}. 
\par
This result is proved by  passing to the limit in a time-discretization scheme, 
where discrete energetic solutions are constructed via time-incremental minimization of a functional 
involving the sum of the dissipation potential $\calR$ and the energy $\calE$. The existence of minimizers 
follows from the \emph{direct method}, provided that the 
energy functional $\calE$ complies with a standard coercivity requirement, cf.\ \eqref{E1} ahead. 
It is shown that 
the discrete solutions fulfill the stability condition
 and  a discrete energy \emph{inequality}. 
 The proof of the discrete stability relies on the fact that the dissipation distance $\calD$ induced by $\calR$
complies with the triangle inequality, see \eqref{D1} later on. 
From the discrete energy inequality  all a priori estimates are derived. For this, 
a crucial role is played by a condition ensuring that the power of the external forces 
$\partial_t \calE$ is controlled by the energy $\calE$ itself, cf.\  \eqref{E2} below,
so that the last integral term on the right-hand side of \eqref{gE} is estimated 
in terms of the energy, and  Gronwall's lemma can be applied. 
 All in all, 
$\calE$ has 
to satisfy the following properties: 
\begin{subequations}
\label{CE}
\begin{align}
\label{E1}
&\begin{array}{l}
\text{\it Compactness of energy sublevels:}\quad
\forall\, t\!\in\! [0,T]\; \forall\, E\!\in\!\R:\\
\qquad\quad
L_{E}(t):=\{\mathbf{q}\in\calQ\,|\,\calE(t,\mathbf{q})\leq E\}\text{ is weakly seq.\ compact.}
\end{array}
\\[1.5mm]
\label{E2}
&\begin{array}{l}
\text{\it Uniform control of the power:}\\
\qquad\quad\exists\, c_0\!\in\!\R\;\exists\,c_1\!>\!0\;\forall\, 
(t,\mathbf{q}) \!\in\![0,T]\!\times\!\calQ
\text{ with }\calE(t, \mathbf{q})<\infty:\\
\qquad\quad\calE(\cdot,\mathbf{q})\in C^1([0,T])\text{ and }
|\partial_t\calE(t,\mathbf{q})|\leq c_1(c_0{+}\calE(t,\mathbf{q})). 
\end{array}
\end{align}
\end{subequations}
\begin{remark}
\upshape
\label{rmk:needed?}
 Observe that   condition \eqref{E2} 
 in fact guarantees a 
Lipschitz estimate for $\calE$ with respect to time via 
Gronwall's lemma, namely  
\begin{align}
\label{Lip}
|\calE(t,\mathbf{q})-\calE(s,\mathbf{q})|\leq\left({\mathop{\mathrm{e}}}^{c_1|t-s|}-1\right)
(\calE(t,\mathbf{q})+c_0)
\leq{\mathop{\mathrm{e}}}^{c_1T}(\calE(t,\mathbf{q})+c_0)|t-s|\,.
\end{align}
Hence, if $\calE(t,\mathbf{q})<E$ for $E\in\R,$ then, for 
$c_E:={\mathop{\mathrm{e}}}^{c_1T}(E+c_0),$ estimate \eqref{Lip} implies
\begin{align}
\label{L}
|\calE(t,\mathbf{q})-\calE(s,\mathbf{q})|\leq c_E|t-s|\,.
\end{align} 
\end{remark}
 As for the dissipation potential $\calR$, we require that the induced 
dissipation distance  
 \begin{equation}
 \label{D-defined}
 \calD:\calZ\times\calZ \to [0,\infty]  \quad \text{ given by }    
\quad \calD(z,\tilde z):=\calR(\tilde z{-}z)  \quad \text{
for all $z,\tilde z\in\calZ$}, 
\end{equation}
fulfills 
\begin{subequations}
\label{CD}
\begin{align}
\label{D1} 
&\begin{array}{ll}
\text{\it Quasi-distance:}\quad
\forall\, z_1,z_2,z_3\in\calZ:\;&
\calD(z_1,z_2)=0\;\Leftrightarrow z_1=z_2\ \text{ and }\\ 
&\calD(z_1,z_3) \leq\calD(z_1,z_2)+\calD(z_2,z_3);
\end{array}
\\[1mm]
\label{D2}
&\begin{array}{l}
\text{\it Semi-continuity:}\\
\quad\qquad
\calD:\;\calZ\times\calZ \to [0,\infty]\text{ is
weakly sequentially lower semi-continuous.}
\end{array}
\end{align}
\end{subequations}
\par
  The abstract existence proof then consists in  passing to the limit in the discrete energy 
inequality by lower semicontinuity arguments, 
leading to an upper energy estimate,  and in the discrete stability condition, 
leading to \eqref{gS}. The lower energy estimate which ultimately  yields the energy balance 
\eqref{gE} then follows from a by now classical procedure, based on the combination 
of the previously proved \eqref{gS} with a Riemann-sum argument. 
For the limit passage  in the discrete energy inequality and in the discrete  stability,
the following compatibility conditions are required: 
\\
{\sl 
For every stable sequence $(t_k,\mathbf{q}_k)_{k\in \N}$ with $t_k\to t,$ 
$\mathbf{q}_k\rightharpoonup \mathbf{q}$ 
{in }$[0,T]\times\calQ${ we have}} 
\\[-6mm]
\begin{subequations}
\label{CC}
\begin{align}
\label{CCa}
\text{Convergence of the power of the energy: }\;&\partial_t \calE(t_k,\mathbf{q}_k) 
\to \partial_t\calE(t,\mathbf{q})\,,\\
\label{CCb}
\text{Closedness of sets of stable states: }\;&\mathbf{q}\in \calS(t)\,. \\[-6mm]
\nonumber
\end{align}
\end{subequations}
\par 
With these prerequisites at hand the abstract existence result reads as follows:  
\begin{teor}[Abstract main existence theorem\;\cite{Miel08DEMF}]
\label{amet}
Let the rate-independent system \\ $(\calQ,\calE,\calR)$ satisfy 
conditions
\eqref{CE} and \eqref{CD}. Moreover, let  the
compatibility conditions \eqref{CC} hold. 
\par
Then, for each ${\mathbf{q}}_0\in \calS(0)$ there exists an energetic solution 
$\mathbf{q}:[0,T]
\to \calQ$ for $(\calQ,\calE,\calR)$  satisfying $\mathbf{q}(0)={\mathbf{q}}_0$. 
\end{teor}  
\par 
 From Thm.\ \ref{amet} we will derive our own existence result for the rate-independent 
system with damage and plasticity.  
\begin{teor}[Existence of energetic solutions for the rate-independent system from \eqref{defFctsfinal}]
\label{EnSolanisoEx}
Let \\ the assumptions \eqref{ass-dom} and \eqref{given}--\eqref{defH} stated in Sec.\ \ref{Ass} be satisfied. 
Then the rate-independent system for damage and plasticity $(\calQ,\calE,\calR)$ given by \eqref{defFctsfinal} 
satisfies the properties \eqref{CE}, \eqref{CD} \& \eqref{CC},  
and hence, for each ${\mathbf{q}}_0\in\calQ$ with ${\mathbf{q}}_0\in\calS(0)$ 
it admits an energetic solution in the sense of 
Def.\ \ref{defenfo}.  
\end{teor}
%
\section{Proof of Theorem \ref{EnSolanisoEx}} 
\label{AnaProp}
%
It is immediate to observe that the  dissipation distance generated by the potential $\calR$ from \eqref{defR} 
via formula \eqref{D-defined} 
satisfies the abstract condition \eqref{CD}. Thus, it remains to verify 
that the energy functional $\calE$ from \eqref{defE}
satisfies the basic properties \eqref{CE}. In addition, the compatibility conditions \eqref{CC} 
have to be deduced.   
\par
To this aim, we start with verifying the following regularity property for the inelastic strain.
\begin{lemm}
\label{ContXi}
Let \eqref{propXi} and \eqref{assq} hold true, let 
$\alpha\in[1,\infty)$. Then, 
$\Xi: 
\Spm 
\times L^2(\Omega;\Sps)\to L^2(\Omega;\Sps)$ is continuous wrt.\ the 
$L^{\alpha}(\Omega) \times  L^2(\Omega;\Sps)$-topology. 
\end{lemm}
\begin{proof}
Consider $  (\chi_k,D_k)_k
 \subset\Spm\times L^2(\Omega;\Sps)$ such that  
$(\chi_k,D_k)_k\to(\chi,D)$ in 
$L^\alpha(\Omega)\times L^2(\Omega;\Sps)$. Hence, up to a subsequence we find that 
   $(\chi_k,D_k)_k\to(\chi,D)$  
pointwise a.e.\ in $\Omega$. Thanks to \eqref{CharaXi} we find that 
$|\Xi(\chi_k,D_k)|^2\to |\Xi(\chi,D)|^2$ 
pointwise a.e.\ in $\Omega$. 
Moreover, \eqref{monoXi} ensures  
$|\Xi(\chi_k,D_k)|^2\leq|\Xi(0,D_k)|^2=|D_k|^2$ for all $k\in\N,$ 
which thus serves as a convergent majorant. 
Hence, $\Xi(\chi_k,D_k) \to \Xi(\chi,D) $ in $L^2(\Omega;\R^{d\times d})$  
by the dominated convergence theorem. 
\end{proof}
The above continuity property is an important ingredient for the verification of 
the following properties of the functional $\calW:$ 
\begin{lemm}[Properties of $\calW$]
\label{BPropsW}
The functional  $\calW: [0,T]  \times  \Sph\times \Spm \times L^{2}(\Omega;\Sps) \to\R$ 
from \eqref{defW} 
has the following properties: 
\begin{eqnarray}
\nonumber 
\text{bound from below: }\hspace*{-5mm}&&\exists\,c_W,\tilde c_W,C_W>0,\,
\forall(t,\uu,\chi,D)\in[0,T]\times \Sph \times (\Spx\cap\Spm)\times\Spu:\qquad\\
\label{Wbd}
&&\calW(t,\uu,\chi,D)\geq c_W\|\uu\|_\Sph^2
-\tilde c_W\|D\|_{L^2(\Omega;\R^{d\times d})}^2-C_W\,,\\
\nonumber 
\hspace*{10mm}
\text{lower semicontinuity: }\hspace*{-5mm}&&\calW\text{ is lower semicontinuous wrt.\ the  weak convergence in }
\Sph\\
\label{Wlsc}
&&\text{and strong convergence in }L^\alpha(\Omega)\times L^2(\Omega;\Sps) 
\text{ for all } \alpha \in [1,\infty). \qquad
\end{eqnarray} 
Moreover, for all 
$(t,\uu,\chi,D)\in [0,T] \times\Sph\times\Spx\cap\Spm\times(\Spu\cap\Spu_1)$ 
the partial time-derivative $\partial_t\calW$ 
is given by 
\begin{equation}
\label{dtW}
\partial_t\calW(t,\uu,\chi,D):=\int_\Omega\big(e(\uu)+e_\mathrm{D}(t)-\Xi(\chi,D)\big):\mathbb{K}(\chi):
\partial_t e_\mathrm{D}(t) 
\,\mathrm{d}x-\langle F_t(t),\uu\rangle_\Sph
\end{equation}
and $\calW+\calG$ satisfies relation \eqref{E2}. 
\end{lemm}
\begin{proof}
We split the proof in several steps.

{\bf Ad bound from below \eqref{Wbd}: } 
Thanks to the positive definiteness of $\mathbb{K},$ the bounds 
on the we given data $F,g,$ cf.\ \eqref{given}, and the properties of $\Xi$ from \eqref{propXi}, 
also using Young's and Korn's inequality  as well as estimate \eqref{growthXi},    
we find for all 
$(t,\uu,\chi,D)\in [0,T] \times \Sph\times(\Spx\cap\Spm)\times(\Spu\cap\Spu_1)$ 
\begin{eqnarray*}
\calW(t,\uu,\chi,D)&\geq&\|e(\uu)+e_\mathrm{D}-\Xi(\chi,D)\|_{L^2(\Omega;\R^{d\times d})}^2
-C_F\|\uu\|_\Sph\\
&\geq&\big(\|e(\uu)\|_{L^2}-(C_\mathrm{D}+\|\Xi(\chi,D)\|_{L^2})\big)^2-C_F\|\uu\|_\Sph\\
&\stackrel{\eqref{growthXi}}{\geq} &\tfrac{1}{2}\|e(\uu)\|_{L^2}^2-8(C_\mathrm{D}+\|D\|_{L^2})^2-C_F\|\uu\|_\Sph\\
&\geq&\tfrac{C_\mathrm{K}^2}{2}\|\uu\|_{\Sph}^2-16(C_\mathrm{D}^2+\|D\|_{L^2}^2)
-\tfrac{C_\mathrm{K}^2}{4}\|\uu\|_\Sph^2-\tfrac{4}{C_\mathrm{K}^2}C_F^2\,,
\end{eqnarray*}
using the short-hand $\|\cdot\|_{L^2}$ for $\|\cdot\|_{L^2(\Omega;\R^{d\times d})}$, and 
with $C_\mathrm{K}$ the constant in Korn's inequality.  This proves \eqref{Wbd}.  
\par 
{\bf Ad lower semicontinuity \eqref{Wlsc}: }
For every $t\in[0,T]$ we observe that $\calW(t,\cdot,\cdot,\cdot)$ 
is continuous wrt.\  the strong convergence in 
$L^2(\Omega;\R^{d\times d})\times L^\alpha(\Omega)\times L^2(\Omega;\Sps)$, 
also due to estimate \eqref{growthXi}. 
Moreover, for every $(t,\chi,D)$ in $[0,T]\times (L^\alpha(\Omega)\cap\Spm)\times L^2(\Omega;\Sps)$ 
the functional $\calW(t,\cdot, \chi,D):\Sph\to\R$ is convex and we have 
$\calW(t, \uu, \chi,D)\geq
-  K_1   \int_\Omega |D| | e(\uu)+e_\mathrm{D}(t) |  \mathrm{d}x-\langle F(t),\uu(t)\rangle_\Sph$. 
Hence, \cite[p.\ 492, Thm.\ 7.5]{FoLeo07} guarantees 
the lower semicontinuity statement \eqref{Wlsc}.   
\par 
{\bf Ad \eqref{dtW} \& relation \eqref{E2}: }
Formula  \eqref{dtW} ensues from a direct calculation, taking into account that
$\calW(t, \cdot,\chi,D)$  is Fr\'echet-differentiable in $L^2(\Omega;\R^{d})$,  as well as  
the regularity properties of $F $ and $\uu_\mathrm{D},$ cf.\ \eqref{given}. 
Note now that $\partial_t(\calW+\calG)=\partial_t\calW$. 
In order to find the bound \eqref{E2} on $|\partial_t\calW(t,\uu,\chi,D)|$ we make use 
of the growth properties of 
$\mathbb{K},$ cf.\ \eqref{given}, H\"older's and Young's 
inequality, and exploit the already deduced bound \eqref{Wbd}. This yields 
\begin{eqnarray*}
\big|\partial_t\calW(t,\uu,\chi,D)\big|
&\leq& 
K_2C_\mathrm{D}\big(\|e(\uu)\|_{L^2}^2+ \|e_\mathrm{D}\|_{L^2}^2 
+\|\Xi(\chi,D)\|_{L^2}^2+2\big)
+\tfrac{1}{2}(C_F^2+\|\uu\|_\Sph^2)\\
&\leq& C\|\uu\|_\Sph^2+c\|D\|_{L^2}^2+C_3\\
&\leq& \tilde C\calW(t,\uu,\chi,D)+\tilde c\|D\|_{L^{q_1}}^{q_1}+C_4\\
&\leq& C_5\big(\calW(t,\uu,\chi,D)+\calG(D)\big)+C_6\,, 
\end{eqnarray*}
where the last estimate follows from the coercivity properties of $\calG$, cf.\ \eqref{Gbdd}. 
This  finishes the proof of \eqref{E2}. 
\end{proof}
\begin{lemm}[Weak lower semicontinuity]
\label{Genwlsc}
Let $\mathbf{B}_1\subset\mathbf{B}_2$ with a continuous embedding 
be separable Banach spaces and $\mathbf{B}_1$ reflexive. 
Assume that the functional $\calE_1:\mathbf{B}_1\to\R$ is weakly sequentially 
lower semicontinuous and coercive. Then the extended functional $\calE_2:\mathbf{B}_2\to\R$ is also 
weakly sequentially lower semicontinuous, where 
\begin{equation}
\calE_2(v):=\left\{
\begin{array}{ll}
\calE_1(v)&\text{if }v\in \mathbf{B}_1,\\
\infty&\text{if }v\in \mathbf{B}_2\backslash\mathbf{B}_1.
\end{array}
\right. 
\end{equation}
\end{lemm}
\begin{proof}
Consider a sequence $v_k\rightharpoonup v$ in $\mathbf{B}_2$. 
If $v_k\in\mathbf{B}_2\backslash\mathbf{B}_1$ for $k\in\N$ except of a 
finite number of indices, then $\calE_2(v)\leq\liminf_{k\to\infty}\calE_2(v_k)=\infty$. 
Also, if $\|v_k\|_{\mathbf{B}_1}\to\infty$ for any subsequence, 
then $\calE_2(v)\leq\liminf_{k\to\infty}\calE_2(v_k)=\infty$. 
Hence, in these two cases there is nothing to prove. 
Thus, assume that $\|v_k\|_{\mathbf{B}_1}\leq C$ for 
a not relabelled subsequence $(v_k)_k$ and a constant $C>0$. 
From the reflexivity of $\mathbf{B}_1$ we now conclude that there is a further, 
not relabelled subsequence and an element $\tilde v\in\mathbf{B}_1$ such that 
$v_k\rightharpoonup \tilde v$ in $\mathbf{B}_1$. By the uniqueness of 
the limit in $\mathbf{B}_2\supset\mathbf{B}_1$ we have that $\tilde v=v$. 
Now the weak lower semicontinuity of $\calE_1:\mathbf{B}_1\to\R$ implies that 
$\calE_2(v)=\calE_1(v)\leq\liminf_{k\to\infty}\calE_1(v_k)\leq\liminf_{k\to\infty}\calE_2(v_k),$ 
which proves the assertion.   
\end{proof}
\begin{lemm}[Properties of $\calH,$ $\calJ,$ and $\calG$]
\label{PropsHJG}
Let the coupling term $\calH$ be defined as in \eqref{defH} and let $\alpha\in[1,\infty)$. 
Then $\calH$ has the following properties: 
\begin{subequations}
\label{propsH}
\begin{eqnarray}
\label{Hdom}
\text{proper domain: }\hspace*{-5mm}&&\mathop{\mathrm{dom}}\calH=\Spm\times\Spv\,,\\
\label{Hbdd}
\text{bound from below: }\hspace*{-5mm}&&\forall(\chi,D)\in\Spz\times\Spv:\;
\calH(\chi,D)\geq 0\,,\\
 \label{Hlsc}
\hspace*{10mm}
\text{lower semicontinuity: }\hspace*{-5mm}&&\calH:\Spz\times\Spv\to[0,\infty]
\text{ is lower semicontinuous wrt.\ the}\\
\nonumber 
&&\text{strong convergence in }L^\alpha(\Omega)\times L^{q_2}(\Omega,\R^{d\times d})
\; \text{for every } \alpha \in [1,\infty)\,.  
\end{eqnarray}
\end{subequations}
Let the damage regularization $\calJ$ be given by \eqref{defJ}. Then, the following properties hold true: 
\begin{subequations}
\label{propsJ}
\begin{eqnarray}
\label{domJ}
\text{proper domain: }\hspace*{-5mm}&&\mathop{\mathrm{dom}}\calJ=(\Spx\cap\Spm)\,,\\
\label{Jbdd}
\hspace*{10mm}
\text{bound from below: }\hspace*{-5mm}&&\exists\, c_\calJ,C_\calJ>0,\,\forall\chi\in(\Spx\cap\Spm):\;
\calJ(\chi)\geq c_\calJ\|\chi\|_\Spx^r-C_\calJ\,,\\
 \label{Jlsc}
\text{lower semicontinuity: }\hspace*{-5mm}&&\calJ\text{ is lower semicontinuous wrt.\  the weak convergence in }
\Spx\,.
\end{eqnarray}
\end{subequations}
Let the plastic regularization $\calG$ be given as in \eqref{defG}. Then, the following properties are satisfied: 
\begin{subequations}
\label{propsG}
\begin{eqnarray}
\label{domG}
\text{proper domain:}\hspace*{-3mm}&&\mathop{\mathrm{dom}}\calG\subset(\Spu\cap\Spu_1)
\text{ is a closed, convex subset}\,,\\
\label{Gbdd}
\text{bound from below:}\hspace*{-3mm}&&\exists\,c_\calG,C_\calG>0,\,
\forall D\in\mathop{\mathrm{dom}}\calG:\;
\calG(D)\geq c_\calG \|D\|_\Spu^q-C_\calG\,,\\
&&
\nonumber
\calG(D)\geq c_\calG\|D\|_{\Spu_1}^{q_1}-C_\calG\,,\\ 
\label{Glsc}
\hspace*{7mm}
\text{lower semicontinuity:}\hspace*{-3mm}
&&\calG  \text{ is lower semicontinuous wrt.\ weak convergence in }
(\Spu\cap\Spu_1)\,.
\end{eqnarray}
\end{subequations}
Hence, the functionals 
$\calJ:\Spz\to\R_\infty$  
and $\calG:\Spv\to\R_\infty$ 
are weakly sequentially 
lower semicontinuous. 
\end{lemm}
\begin{proof}
We split the proof in several steps.
\par 
{\bf Ad \eqref{propsH}: }The domain property \eqref{Hdom} and the boundedness from below are a direct 
consequence of definition \eqref{defH1} and \eqref{growthH}. The lower semicontinuity can be concluded from 
the continuity \eqref{contiH} and the growth property \eqref{growthH} as follows. 
Given a sequence $(\chi_k,D_k)_k\subset(\Spz\backslash\Spm)\times\Spv$, with 
$(\chi_k,D_k)\to(\chi,D)$ in $L^\alpha(\Omega)\times  L^{q_2}(\Omega,\R^{d\times d})$, 
we immediately find that 
$\infty = \liminf_{k\to\infty}\calH(\chi_k,D_k)\geq\calH(\chi,D)$, while  no matter occurs 
if $(\chi,D)\in\Spm\times\Spv$. 
Hence assume that there is a (not relabelled) subsequence $(\chi_k,D_k)_k\subset\Spm\times\Spv$ 
such that $(\chi_k,D_k)\to(\chi,D)$ in $L^\alpha(\Omega)\times  L^{q_2}(\Omega,\R^{d\times d})$. 
Upon extraction of a further subsequence that converges pointwise a.e.\ in $\Omega$ we find that 
the limit $(\chi,D)\in\Spm\times\Spv$. Moreover, thanks to the continuity \eqref{contiH} 
we have that $H(\chi_k,D_k)\to H(\chi,D)$ a.e.\ in $\Omega$ along this subsequence. In addition, 
for each $k\in\N,$  
the growth property \eqref{growthH} guarantees the convergent majorant 
 $C_H(|D|^{q_2}+1)\geq H(\chi_k,D_k),$  so that the convergence of the respective integral terms is 
implied by the dominated convergence theorem. Thus, altogether, we have verified the lower 
semicontinuity property stated in \eqref{Hlsc}.      
\par 
{\bf Ad \eqref{propsJ} \& \eqref{propsG}: }
Properties \eqref{domJ} and \eqref{domG} are implied by \eqref{defJ} and \eqref{defG}, respectively. 
The bounds \eqref{Jbdd} and \eqref{Gbdd} immediately follow from the growth properties 
\eqref{growthJ} and \eqref{growthG}. Invoking \cite[p.\ 492, Thm.\ 7.5]{FoLeo07}, 
the latter, together with the continuity \eqref{contiJ}, 
resp.\ \eqref{contiG} and the convexity property \eqref{convexJ}, resp.\ \eqref{convexG}, also ensure the 
weak sequential lower semicontinuity.  
\par 
The last statement of the lemma follows from \eqref{Jlsc} and \eqref{Glsc} as a direct 
consequence of Lemma \ref{Genwlsc}.   
\end{proof}
We are now in the position to conclude properties \eqref{CE} for $\calE$ 
in consequence of Lemmata \ref{ContXi}--\ref{PropsHJG}.  
\begin{coro}[Properties \eqref{CE} of $\calE$]
\label{BPropsE}
Let the functional $\calE:[0,T]\times\Sph\times\Spz\times\Spv\to\R_\infty$ 
be defined as in \eqref{defE}. Let the assumptions of Lemmata \ref{ContXi}--\ref{PropsHJG} 
hold true. 
Then, the functional $\calE$ satisfies properties \eqref{CE}.  
\end{coro}
\begin{proof}
We split the proof in several steps.
\par 
{\bf Ad compactness of the sublevels \eqref{E1}: }
Comparing \eqref{Wbd} with \eqref{Gbdd}, using H\"older's and Young's inequality, 
we first deduce for $D$: 
\begin{equation*}
\begin{aligned}
-\bar c_W\|D\|_{L^2(\Omega,\R^{d\times d})}^2 & 
\geq-\|D\|_{\Spu_1}^2\bar c_W \calL^d(\Omega)^{(q_1-2)/q_1}
\\
& 
\geq-\tfrac{2c_\calG}{q_1}\|D\|_{\Spu_1}^{q_1}
-\tfrac{q_1-2}{q_1}\big(c_\calG^{-2/q_1}\bar c_W\calL^d(\Omega)^{(q_1-2)/q_1}\big)^{(q_1-2)/q_1}\,,
\end{aligned}
\end{equation*}
where $2/q_1<1$ according to \eqref{growthG}. 
Thus, combining bounds \eqref{Wbd}, \eqref{Hbdd}, \eqref{Jbdd}, and \eqref{Gbdd} 
yields that $\calE$ has sublevels bounded in $\Sph \times \Spx \times (\Spu \cap \Spu_1)$.  
Since this space is reflexive, 
the sublevels are then  sequentially  weakly  compact, and so they are 
 in $\calQ = \Sph \times \Spz \times \Spv$.  
 %
\par 
{\bf Ad uniform control of the power \eqref{E2}: }
Since $\partial_t\calE(t,\uu,\chi,D)=\partial_t\calW(t,\uu,\chi,D),$ given by \eqref{dtW}, 
the last statement of Lemma \ref{BPropsW} ensures \eqref{E2} for $\calE$ upon adding $\calH(\chi,D)>0$ 
and $\calJ(\chi)+C_\calJ >0 $.     
\end{proof}
\par 
Finally, we verify that the rate-independent system $(\calQ,\calE,\calR)$ for plasticity and damage 
satisfies the compatibility conditions \eqref{CC}.
\begin{prop}[Compatibility conditions \eqref{CC}] 
Let the assumptions of Theorem \ref{EnSolanisoEx} hold true. 
Then the rate-independent system $(\calQ,\calE,\calR)$ for plasticity and damage 
satisfies the compatibility conditions \eqref{CC}. 
\end{prop}
\begin{proof} 
In view of \eqref{L}   
 we infer for any stable sequences 
$(t_k,\uu_k,\chi_k,D_k)_k\subset\calS(t)$ that there is a constant $E>0$ such that 
this sequence belongs to same the energy sublevel $L_E(t)$, 
which is bounded in $\Sph \times \Spx \times (\Spu \cap \Spu_1)$
as  guaranteed by  
Cor.\ \ref{BPropsE}. Hence,  
we deduce 
the following convergence properties  along a (not relabelled) subsequence: 
\begin{subequations}
\label{conv}
\begin{eqnarray}
\uu_k\rightharpoonup\uu&&\text{in }\Sph\,,\\
D_k\rightharpoonup D&&\text{in }\Spu\,,\\
D_k\to D&&\text{in } L^{q_1}(\Omega,\R^{d\times d})\cap L^{q_2}(\Omega,\R^{d\times d})\,,\\
\chi_k\rightharpoonup\chi&&\text{in }\Spx\,,\\
\chi_k\to\chi&&\text{in }L^\alpha(\Omega)\text{ for any }\alpha\in[1,\infty)\,.
\end{eqnarray}
\end{subequations}    

{\bf Ad \eqref{CCa} convergence of the power $\partial_t\calE(t_k,q_k)$: }
In view of the above convergences, property 
\eqref{CCa} can be concluded from weak-strong 
convergence arguments using that $e_\mathrm{D}(t_k)\to e_\mathrm{D}(t)$ strongly in $\Sph$ and 
$F(t_k)\to F(t)$ strongly in $\Sph^*$ thanks to the regularity assumptions \eqref{given}.   

{\bf Ad closedness of sets of stable states \eqref{CCb}: } 
In order to deduce \eqref{CCb}, we make use of the so-called mutual recovery condition, i.e.\ 
\emph{for every sequence $(\uu_k,\chi_k,D_k)_k\subset\calS(t)$ converging to a limit 
$(\uu,\chi,D)$ in the sense of \eqref{conv}, and any competitor $(\hat{\uu},\hat\chi,\hat D),$ 
it must be possible to construct a mutual recovery sequence $(\hat{\uu}_k,\hat\chi_k,\hat D)_k$ 
such that 
\begin{equation}
\label{mrs-condition}
\limsup_{k\to\infty}\big(\calE(t,\hat q_k)-\calE(t,q_k)+\calR(\hat q_k-q_k)\big)
\leq \calE(t,\hat q)-\calE(t,q)+\calR(\hat q-q)\,,
\end{equation}
where we again abbreviated $\hat q_k=(\hat{\uu}_k,\hat\chi_k,\hat D_k),$ etc.. }
\par
Let $\hat q=(\hat{\uu},\hat\chi,\hat D)$ such that $\calE(t,\hat q)<\infty$.  
Then, a suitable recovery sequence is defined by 
\begin{subequations}
\label{mrs}
\begin{eqnarray}
\label{mrsu}
\hat{\uu}_k&:=&\hat{\uu}\,,\\
\label{mrschi}
\hat\chi_k&:=&\min\{\chi_k,\max\{0,\hat\chi-\delta_k\}\}\,,\\
\label{mrsD}
\hat D_k&:=&\hat D\,,
\end{eqnarray}
\end{subequations}
where $\delta_k$ in \eqref{mrschi} is suitably chosen in dependence of $\|\chi_k-\chi\|_\Spx$ 
such that $\delta_k\to0$ as $k\to\infty,$ see \cite{ThoMie09DNEM} for the details. 
We refer to \cite[Thm.\ 4.5]{LRTT14RDTM} for the proof of the following convergence property:   
\begin{equation}
\label{propmrschi}
\hat\chi_k\rightharpoonup\hat\chi\;\text{ in }\Spx
\;\text{ as well as }\;
\limsup_{k\to\infty}\big(
\calJ(\hat\chi_k)-\calJ(\chi_k)
\big)
\leq \calJ(\hat\chi)-\calJ(\chi)\,.
\end{equation}
The convergence stated in \eqref{propmrschi} together with \eqref{mrsD} 
yields that 
$\calR(\hat\chi_k-\chi_k,\hat D_k-D_k)
\to \calR(\hat\chi-\chi,\hat D-D)$. 
Moreover, upon choosing a further subsequence $(\hat{\uu}_k,\hat\chi_k,\hat D_k)_k$, 
which converges pointwise a.e.\ in $\Omega,$ and by making use of the bounds \eqref{given}, 
we may conclude via the dominated convergence theorem, also taking into 
account the growth properties  \eqref{growthH} of $\calH$,
that  
\begin{equation*}
\calW(t, \hat{\uu}_k, \hat\chi_k,\hat D_k) \to \calW(t,\hat \uu, \hat\chi, \hat D) \text{ and } 
\calH(\hat\chi_k, \hat{D}_k)\to  \calH(\hat\chi, \hat{D}),
\end{equation*}
whereas we clearly have $\calG (\hat{D}_k) \to \calG (\hat D)$. 
The respective expressions for $(\uu_k,\chi_k,D_k)$ can be handled by weak lower semicontinuity.
Ultimately, we conclude \eqref{mrs-condition}, which finishes the proof. 
\end{proof}
\bibliographystyle{alpha}
\bibliography{marita_lit.bib,anisotr_dam_biblio.bib}

\begin{thebibliography}{DMDMM08}

\bibitem[AKS10]{FiKnSt10YMQS}
Fiaschi A., D.~Knees, and U.~Stefanelli.
\newblock Young-measure quasi-static damage evolution.
\newblock {\em J.\ Math.\ Anal.\ Appl.}, 401:269--288, 2010.

\bibitem[AMV14]{AlMaVi14GDMC}
R.~Alessi, J.J. Marigo, and S.~Vidoli.
\newblock Gradient damage models coupled with plasticity and nucleation of
  cohesive cracks.
\newblock {\em Archive for Rational Mechanics and Analysis}, 214(2):575--615,
  2014.

\bibitem[BMR09]{BouMiRou07}
G.~Bouchitt{\'e}, A.~Mielke, and T.~Roub{\'i}{\v c}ek.
\newblock A complete-damage problem at small strain.
\newblock {\em Zeit. angew. Math. Phys.}, 60:205--236, 2009.

\bibitem[BMR12]{BaMiRo12QSSP}
S.~Bartels, A.~Mielke, and T.~Roub{\'i\v c}ek.
\newblock Quasistatic small-strain plasticity in the limit of vanishing
  hardening and its numerical approximation.
\newblock {\em SIAM J.\ Numer.\ Anal.}, 50(2):951--976, 2012.

\bibitem[BR08]{BaRo08TVPS}
S.~Bartels and T.~Roub{\'i\v c}ek.
\newblock Thermoviscoplasticity at small strains.
\newblock {\em ZAMM Z.\ Angew.\ Math.\ Mech.}, 2008.

\bibitem[BS04]{BoSch}
E.~Bonetti and G.~Schimperna.
\newblock Local existence for {F}r\'emond's model of damage in elastic
  materials.
\newblock {\em Contin. Mech. Thermodyn.}, 16(4):319--335, 2004.

\bibitem[BSS05]{bss}
E.~Bonetti, G.~Schimperna, and A.~Segatti.
\newblock On a doubly nonlinear model for the evolution of damaging in
  viscoelastic materials.
\newblock {\em J. Differential Equations}, 218(1):91--116, 2005.

\bibitem[CL15]{Crismale-Lazzaroni}
V.~Crismale and G.~Lazzaroni.
\newblock Viscous approximation of quasistatic evolutions for a coupled
  elastoplastic-damage model.
\newblock {\em Preprint SISSA 05/2015/MATE}, 2015.

\bibitem[Cri]{Crismale}
V.~Crismale.
\newblock Globally stable quasistatic evolution for a coupled
  elastoplastic�damage model.
\newblock {\em ESAIM Control Optim. Calc. Var.}
\newblock to appear.

\bibitem[DMDM06]{DMDSMo06QEPL}
G.~Dal~Maso, A.~DeSimone, and M.G. Mora.
\newblock Quasistatic evolution problems for linearly elastic-perfectly plastic
  materials.
\newblock {\em Arch.\ Rational Mech.\ Anal.}, 180:237--291, 2006.

\bibitem[DMDMM08]{DMDSMoMo08GSQE}
G.~Dal~Maso, A.~DeSimone, M.G. Mora, and M.~Morini.
\newblock Globally stable quasistatic evolution in plasticity with softening.
\newblock {\em Netw.\ Heterog.\ Media}, 3:567--614, 2008.

\bibitem[DMI13]{DMIur13FMGL}
G.~Dal~Maso and F.~Iurlano.
\newblock Fracture models as {G}amma-limits of damage models.
\newblock {\em Commun.\ Pure Appl.\ Anal.}, pages 1657--1686, 2013.

\bibitem[DMS14]{DMSca14QEPP}
G.~Dal~Maso and R.~Scala.
\newblock Quasistatic evolution in perfect plasticity as limit of dynamic
  processes.
\newblock {\em J.\ Dynam.\ Differential Equations}, 26:915--954, 2014.

\bibitem[DPO93]{DP-Owen}
Gianpietro Del~Piero and David~R. Owen.
\newblock Structured deformations of continua.
\newblock {\em Arch. Rational Mech. Anal.}, 124(2):99--155, 1993.

\bibitem[FG06]{FraGar06VVPB}
G.~Francfort and A.~Garroni.
\newblock A variational view of partial brittle damage evolution.
\newblock {\em Arch. Rational Mech. Anal.}, 182:125--152, 2006.

\bibitem[FL07]{FoLeo07}
I.~Fonseca and G.~Leoni.
\newblock {\em Modern Methods in the Calculus of Variations: $L^p$ Spaces}.
\newblock Springer, 2007.

\bibitem[FN96]{FreNed96DGDP}
M.~Fr{\'e}mond and B.~Nedjar.
\newblock Damage, gradient of damage and principle of virtual power.
\newblock {\em Internat. J. Solids Structures}, 33:1083--1103, 1996.

\bibitem[FRC10]{FRC}
Francesco Freddi and Gianni Royer-Carfagni.
\newblock Regularized variational theories of fracture: a unified approach.
\newblock {\em J. Mech. Phys. Solids}, 58(8):1154--1174, 2010.

\bibitem[Fr{\'e}02]{Fre02}
M.~Fr{\'e}mond.
\newblock {\em Non-Smooth Thermomechanics}.
\newblock Springer-Verlag Berlin Heidelberg, 2002.

\bibitem[Gia05]{Giac05ATAQ}
A.~Giacomini.
\newblock Ambrosio-{T}ortorelli approximation of quasi-static evolution of
  brittle fracture.
\newblock {\em Calc. Var. Partial Differential Equations}, 22:129--172, 2005.

\bibitem[GL09]{GarLar09TBQS}
A.~Garroni and C.~Larsen.
\newblock Threshold-based quasi-static brittle damage evolution.
\newblock {\em Arch. Ration. Mech. Anal.}, 194(2):585--609, 2009.

\bibitem[Hil50]{Hill50}
R.~Hill.
\newblock {\em The mathematical theory of plasticity}.
\newblock Oxford University Press, Oxford-New York, 1950.

\bibitem[HK15]{HeiKra15CDLE}
C.~Heinemann and C.~Kraus.
\newblock Complete damage in linear elastic materials -- modeling, weak
  formulation and existence results.
\newblock {\em Calc. Var. Partial Differ. Equ.}, 54:217--250, 2015.

\bibitem[HR99]{HanRed99PMTN}
W.~Han and B.D. Reddy.
\newblock {\em Plasticity (Mathematical Theory and Numerical Analysis)},
  volume~9 of {\em Interdisciplinary Applied Mathematics}.
\newblock Springer-Verlag, New York, 1999.

\bibitem[JRZ13]{JiRoZe13LAVB}
M.~Jir{\'a}sek, Roko{\v s}, and J.~Zeman.
\newblock Localization analysis of variationally based gradient plasticity
  model.
\newblock {\em Int.\ J.\ Solids and Structures}, (1):256--269, 2013.

\bibitem[JZ15]{JirZem15LSRV}
Jir{\'a}sek and J.~Zeman.
\newblock Localization study of a regularized variational damage model.
\newblock {\em Int.\ J.\ Solids and Structures}, 69--70:131--15, 2015.

\bibitem[Kac90]{Kacha90ICDM}
L.M. Kachanov.
\newblock {\em Introduction to Continuum Damage Mechanics}.
\newblock Mechanics of Elastic Stability. Kluwer Academic Publishers, second
  edition edition, 1990.

\bibitem[Kne09]{Knee09SNGS}
D.~Knees.
\newblock Short note on global spatial regularity in elasto-plasticity with
  linear hardening.
\newblock {\em Calc. Var. Partial Differ. Equ.}, 2009.

\bibitem[Kne10]{Knee10GSRC}
D.~Knees.
\newblock On global spatial regularity and convergence rates for time-dependent
  elasto-plasticity.
\newblock {\em Math. Models Methods Appl. Sci.}, 20(10):1823--1858, 2010.

\bibitem[KRZ13]{KnRoZa13VVAR}
D.~Knees, R.~Rossi, and C.~Zanini.
\newblock A vanishing viscosity approach to a rate-independent damage model.
\newblock {\em Math. Models Methods Appl. Sci.}, 23:565--616, 2013.

\bibitem[LRTR14]{LRTT14RDTM}
G.~Lazzaroni, R.~Rossi, M.~Thomas, and Toader R.
\newblock Rate-independent damage in thermo-viscoelastic materials with
  inertia. wias-preprint 2025.
\newblock 2014.

\bibitem[Lub90]{Lub90Plast}
J.~Lubliner.
\newblock {\em Plasticity theory}.
\newblock Macmillan, New York, 1990.

\bibitem[Mie05]{Miel05ERIS}
A.~Mielke.
\newblock Evolution in rate-independent systems ({C}h.6).
\newblock In C.M. Dafermos and E.~Feireisl, editors, {\em Handbook of
  Differential Equations, Evolutionary Equations, vol.~2}, pages 461--559.
  Elsevier B.V., Amsterdam, 2005.

\bibitem[Mie11a]{Miel11CDEE}
A.~Mielke.
\newblock Complete-damage evolution based on energies and stresses.
\newblock {\em Discrete Contin.\ Dyn.\ Syst.\ Ser.\ S}, 4:423--439, 2011.

\bibitem[Mie11b]{Miel08DEMF}
A.~Mielke.
\newblock {\em Differential, energetic and metric formulations for
  rate-independent processes. In: Nonlinear {PDE}'s and applications}, volume
  2028 of {\em Lecture Notes in Mathematics}.
\newblock Springer, Heidelberg; Fondazione C.I.M.E., Florence, 2011.
\newblock Notes of the C.I.M.E. Summer School held in Cetraro, June 23--28,
  2008, Edited by Luigi Ambrosio and Giuseppe Savar{\'e}, Centro Internazionale
  Matematico Estivo (C.I.M.E.) Summer Schools.

\bibitem[MM05]{MaiMie05EREM}
A.~Mainik and A.~Mielke.
\newblock Existence results for energetic models for rate-independent systems.
\newblock {\em Calc.\ Var.\ PDEs}, 22:73--99, 2005.

\bibitem[MR06]{MiRou06}
A.~Mielke and T.~Roub{\'i}{\v c}ek.
\newblock Rate-independent damage processes in nonlinear elasticity.
\newblock {\em M$^3$AS Math.\ Models Methods Appl. Sci.}, 16:177--209, 2006.

\bibitem[MRS08]{MRS06}
A.~Mielke, T.~Roub{\'i}{\v c}ek, and U.~Stefanelli.
\newblock {$\Gamma$}-limits and relaxations for rate-independent evolutionary
  problems.
\newblock {\em Calc.\ Var.\ Partial Differ.\ Equ.}, 31:387--416, 2008.

\bibitem[MRZ10]{MiRoZe10CDEV}
A.~Mielke, T.~Roub{\'i\v c}ek, and Zeman.
\newblock Complete damage in elastic and viscoelastic media and its energetics.
\newblock {\em Computer Methods in Applied Mechanics and Engineering},
  (21-22):1242--1253, 2010.

\bibitem[MT04]{MiTh04RIHM}
A.~Mielke and F.~Theil.
\newblock On rate-independent hysteresis models.
\newblock {\em Nonl.\ Diff.\ Eqns.\ Appl.\ (NoDEA)}, 11:151--189, 2004.

\bibitem[PM13]{PhaMar13ODRC}
K.~Pham and J.J. Marigo.
\newblock From the onset of damage to rupture: Construction of responses with
  damage localization for a general class of gradient damage models.
\newblock {\em Continuum Mechanics and Thermodynamics}, 25:147--171, 2013.

\bibitem[RDG08]{RaDeGa08ADMB}
F.~Ragueneau, R.~Desmorat, and F.~Gatuingt.
\newblock Anisotropic damage modelling of biaxial behaviour and rupture of
  concrete structures.
\newblock {\em Computers and concrete}, 5(4):417--434, 2008.

\bibitem[Tem85]{Tema85MPP}
R.~Temam.
\newblock {\em Mathematical Problems in plasticity}.
\newblock Gauthier-Villars, 1985.

\bibitem[Tho13]{Thom11QEBV}
M.~Thomas.
\newblock Quasistatic damage evolution with spatial {BV}-regularization.
\newblock {\em DCDS-S}, 6(1):235--255, 2013.

\bibitem[TM10]{ThoMie09DNEM}
M.~Thomas and A.~Mielke.
\newblock Damage of nonlinearly elastic materials at small strain: existence
  and regularity results.
\newblock {\em Zeit. angew. Math. Mech.}, 90(2):88--112, 2010.

\end{thebibliography}
\end{document}